\documentclass[11pt]{amsart}
\usepackage{amsmath}
\usepackage{amsfonts}
\usepackage{amsthm}
\usepackage{amssymb}
\usepackage{tikz-cd}
\usepackage{graphics}
\usepackage{array}
\usepackage{enumerate}
\usepackage{dsfont}
\usepackage{color}
\usepackage{wasysym}
\usepackage{hyperref}
\usepackage{pdfsync}

\newcommand{\bel}[1]{\begin{equation}\label{#1}}

\newcommand{\be}{\begin{equation}}

\newcommand{\ba}{\begin{eqnarray}}
\newcommand{\ea}{\end{eqnarray}}

\newcommand{\qe}{\end{equation}}
\newcommand{\R}{{\mathbb R}}

\newcommand{\C}{{\mathbb C}}
\newcommand{\Q}{{\mathbb Q}}

\newcommand{\Hmm}[1]{\leavevmode{\marginpar{\tiny%
$\hbox to 0mm{\hspace*{-0.5mm}$\leftarrow$\hss}%
\vcenter{\vrule depth 0.1mm height 0.1mm width \the\marginparwidth}%
\hbox to
0mm{\hss$\rightarrow$\hspace*{-0.5mm}}$\\\relax\raggedright #1}}}

\theoremstyle{plain}

\newtheorem{thm}{Theorem}[section]

\newtheorem{prop}[thm]{Proposition}

\newtheorem{coro}[thm]{Corollary}
\newtheorem{lemma}[thm]{Lemma}

\newtheorem*{AbundCon}{Good Minimal Models Conjecture}
\newtheorem*{NonvanCon}{Nonvanishing Conjecture}

\theoremstyle{definition}

\newtheorem{defi}[thm]{Definition}

\begin{document}

\title[On Nonvanishing for uniruled log canonical pairs]{On Nonvanishing for\\ uniruled log canonical pairs}

	\author{Vladimir Lazi\'c}
	\address{Fachrichtung Mathematik, Campus, Geb\"aude E2.4, Universit\"at des Saarlandes, 66123 Saarbr\"ucken, Germany}
	\email{lazic@math.uni-sb.de}
	
	\author{Fanjun Meng}
	\address{Department of Mathematics, Northwestern University, 2033 Sheridan Road, Evanston, IL 60208, USA}
	\email{fanjunmeng2022@u.northwestern.edu}

	\thanks{
		Lazi\'c was supported by the DFG-Emmy-Noether-Nachwuchsgruppe ``Gute Strukturen in der h\"oherdimensionalen birationalen Geometrie". Meng would like to express his sincere gratitude to his advisor Mihnea Popa for proposing related problems, illuminating him and selflessly supporting him. We would like to thank Mircea Musta{\c{t}}{\u{a}}, Nikolaos Tsakanikas, Charles Stibitz and Yuan Wang for helpful discussions and for suggesting references, and the referees for valuable comments.
	}

\begin{abstract}
We prove the Nonvanishing conjecture for uniruled projective log canonical pairs of dimension $n$, assuming the Nonvanishing conjecture for smooth projective varieties in dimension $n-1$. We also show that the existence of good minimal models for non-uniruled projective klt pairs in dimension $n$ implies the existence of good minimal models for projective log canonical pairs in dimension $n$.
\end{abstract}

\maketitle
	\setcounter{tocdepth}{1}
	\tableofcontents

\section{Introduction}

In this paper we prove several results related to the Nonvanishing conjecture and the existence of good minimal models for projective log canonical pairs over the complex numbers. 

\begin{AbundCon}
Let $(X,\Delta)$ be a projective log canonical pair. If $K_X+\Delta$ is pseudoeffective, then $(X,\Delta)$ has a good minimal model.
\end{AbundCon}

	The conjecture is known to hold in full generality only in dimensions at most $3$. In dimension $3$, it is a culmination of work of many authors, ending in \cite{KMM94,Sho96}. Its important part is:

\begin{NonvanCon}
Let $(X,\Delta)$ be a projective log canonical pair. If $K_X+\Delta$ is pseudoeffective, then there exists an $\R$-divisor $D\geq0$ such that $K_X+\Delta\sim_\R D$.
\end{NonvanCon}

A complementing conjecture is the \emph{Abundance conjecture}, which predicts that for a log canonical pair, any of its minimal models is good, see Section \ref{sec:prelim}.

\medskip

We say that a pair $(X,\Delta)$ is \emph{uniruled} if $X$ is uniruled; and similarly for \emph{non-uniruled} pairs. Our main result proves the Nonvanishing conjecture for uniruled pairs, assuming the Nonvanishing conjecture in lower dimensions.

\begin{thm}\label{main}
The Nonvanishing conjecture for smooth projective varieties in dimension $n-1$ implies the Nonvanishing conjecture for uniruled projective log canonical pairs in dimension $n$.
\end{thm}

The proof is in Section \ref{sec:nonvan}, by using the techniques from \cite{LT19}. This result improves considerably on the previous related results in the literature. If $(X,\Delta)$ is \emph{klt} and $\Delta$ is \emph{rational}, then an analogous statement was proved in \cite[Theorem 8.8]{DHP}, using however a much stronger assumption -- the existence of good minimal models in dimension $n-1$.

An immediate corollary is:

\begin{coro}\label{uni4}
Let $(X,\Delta)$ be a uniruled projective log canonical pair of dimension $4$. If $K_X+\Delta$ is pseudoeffective, then there exists an $\R$-divisor $D\geq0$ such that $K_X+\Delta\sim_\R D$.
\end{coro}

	Another corollary of Theorem \ref{main} is the equivalence between the Nonvanishing conjecture in dimension $n-1$ and the Nonvanishing conjecture for uniruled log canonical pairs in dimension $n$. This is Corollary \ref{equi} below.

\medskip

Our second goal is to discuss relationships between the existence of good minimal models for several classes of pairs. It is known that the termination of flips for klt pairs implies the termination of flips for log canonical pairs, see for instance \cite{Bir07,Fuj07a}. Similarly, the Nonvanishing for klt pairs (even for smooth varieties) implies the Nonvanishing for log canonical pairs, see \cite{Has18} and Theorem \ref{main} above; and the existence of minimal models for klt pairs (even for smooth varieties) implies the existence of minimal models for log canonical pairs \cite{LT19}. Regarding Abundance, statements in a similar direction were obtained in \cite{KMM94,HX16}.

In this context, the following is a very satisfactory result regarding the existence of good minimal models. Note that we say that a pair $(X,\Delta)$ has a \emph{rational boundary} if the coefficients of $\Delta$ are rational numbers and $K_X+\Delta$ is $\Q$-Cartier. 

\begin{thm}\label{main2}
The existence of good minimal models for projective non-uniruled klt pairs with rational boundaries in dimension $n$ implies the existence of good minimal models for projective log canonical pairs in dimension $n$. 
\end{thm}

Working with log canonical pairs is often much more difficult than working with klt pairs, which explains the significance of Theorem \ref{main2}. Additionally, techniques which work for non-uniruled pairs often fail on uniruled pairs, see \cite{LP18,LP20b}. Theorem \ref{main2} improves considerably on \cite[Theorem 1.2]{DL15}, by both removing an assumption and improving the conclusion.

There is an analogous statement in the context of the Abundance conjecture, see Theorem \ref{gm2}, which generalizes \cite[Theorem 1.1]{DL15} to log ca\-no\-ni\-cal pairs.

Most of the results in this paper can be generalised to the relative setting without difficulties. For instance, the relative version of Theorem \ref{main} follows from Theorem \ref{main} and \cite[Lemma 3.2.1]{BCHM}, and the relative version of Theorem \ref{main2} follows from Theorem \ref{main2} and \cite[Theorems 1.2 and 1.7]{HH19}.

\section{Preliminaries}\label{sec:prelim}

Throughout the paper we work over $\C$ and all varieties are normal and projective. A \emph{fibration} is a projective surjective morphism with connected fibres. A \emph{birational contraction} is a birational map whose inverse does not contract any divisors.

\medskip

Given a normal projective variety $ X $ and a pseudoeffective $\R$-Cartier $\R$-divisor $ D $ on $X$, we denote by $\kappa_\iota(X,D)$ the \emph{invariant Iitaka dimension} of $D$, see \cite{Cho08}; if the divisor $D$ is rational, we denote its Iitaka dimension by $\kappa(X,D)$. We denote by $ \nu(X,D) $ the \emph{numerical dimension} of $ D $, see \cite[Chapter V]{Nak04}, \cite{Kaw85}; note that the numerical dimension that we use in this paper was denoted by $\kappa_\sigma$ in \cite{Nak04}. 

We use frequently and without explicit mention that the invariant Iitaka dimension and the numerical dimension behave well under proper pullbacks: if $D$ is an $\R$-Cartier $\R$-divisor on a normal variety $X$, and if $f\colon Y\to X$ is a proper surjective morphism from a normal variety $Y$, then 
$$\kappa_\iota(X,D)=\kappa_\iota(Y,f^*D)\quad\text{and}\quad\nu(X,D)=\nu(Y,f^*D);$$ 
and if, moreover, $f$ is birational and $E$ is an effective $f$-exceptional divisor on $Y$, then
$$\kappa_\iota(X,D)=\kappa_\iota(Y,f^*D+E)\quad\text{and}\quad\nu(X,D)=\nu(Y,f^*D+E).$$

Given a smooth projective variety $ X $ and a pseudoeffective $\R$-divisor $ D $ on $X$, we denote by $P_\sigma(D)$ and $N_\sigma(D)$ the $\R$-divisors forming the \emph{Nakayama--Zariski decomposition} of $D$, see\ \cite[Chapter III]{Nak04}. 

\medskip

For the definitions and basic results on the singularities of pairs and the Minimal Model Program (MMP) we refer to \cite{KM98}. Unless otherwise stated, in a pair $(X,\Delta)$ the boundary $\Delta$ always has real coefficients. A pair $(X,\Delta)$ is \emph{log smooth} if $X$ is smooth and $\Delta$ has simple normal crossings support.

We need the following useful definition.

\begin{defi}
Let $(X, \Delta)$ be a klt pair and $G$ an effective $\mathbb{R}$-Cartier $\mathbb{R}$-divisor such that $K_X+\Delta+G$ is pseudoeffective. Then the \emph{pseudoeffective threshold} of $(X, \Delta)$ with respect to $G$ is 
$$\tau(X, \Delta; G):=\min\{t\in \mathbb{R}\mid K_X+\Delta+tG\text{ is pseudoeffective}\}.$$
\end{defi}

We distinguish between two types of good minimal models: good minimal models in the usual sense and good minimal models in the sense of Birkar--Hashizume. In this paper, the phrase \emph{good minimal model} always means a good minimal model in the usual sense. 

\begin{defi} 
	Let $(X,\Delta)$ be a log canonical pair. A birational map $ \varphi \colon (X,\Delta) \dashrightarrow (Y,\Delta_Y)$ to a $\Q$-factorial log canonical pair $ (Y,\Delta_Y) $ is a \emph{good minimal model in the sense of Birkar--Hashizume} of the pair $(X,\Delta)$ if $ \Delta_Y =\varphi_*\Delta+E$, where $E$ is the sum of all prime divisors which are contracted by $\varphi^{-1}$, if the divisor $K_Y+\Delta_Y$ is semiample and if
	$$a(F,X,\Delta) < a(F,Y,\Delta_Y)$$
	for any prime divisor $ F $ on $ X $ which is contracted by $\varphi $.
		
	If, moreover, the map $\varphi$ is a birational contraction, but $Y$ is not necessarily $\Q$-factorial, then $\varphi$ is a \emph{good minimal model} of $(X,\Delta)$.
\end{defi}

The definitions of \emph{minimal models} and \emph{minimal models in the sense of Birkar--Shokurov} are similar; see for instance \cite[\S2.2]{LT19} for a discussion of their differences.

The following results are used often in the rest of the paper.

\begin{lemma}\label{lem:ratreal}
The existence of good minimal models for klt pairs (respectively non-uniruled klt pairs) with rational boundaries in dimension $n$ implies the exis\-ten\-ce of good minimal models for klt pairs (respectively non-uniruled klt pairs) in dimension $n$. 
\end{lemma}

\begin{proof}
Let $(X,\Delta)$ be a klt pair of dimension $n$. Then it has a minimal model $(X',\Delta')$ by \cite[Corollary 0.3]{BDPP} and \cite[Theorem A]{LT19}. By passing to a small $\Q$-factorialization \cite[Theorem 10.5]{Fuj11}, we may assume that $X'$ is $\Q$-factorial. By \cite[Proposition 3.2(3)]{Bir11} there exist finitely many $\Q$-divisors $\Delta_i$ and positive real numbers $r_i$ such that each pair $(X',\Delta_i)$ is klt, each $K_{X'}+\Delta_i$ is nef and $K_{X'}+\Delta'=\sum r_i(K_{X'}+\Delta_i)$. By assumption, each $K_{X'}+\Delta_i$ is semiample, hence $K_{X'}+\Delta'$ is semiample.
\end{proof}

\begin{lemma}\label{ind}
The existence of good minimal models for klt pairs (respectively non-uniruled klt pairs) in dimension $n$ implies the existence of good minimal models for klt pairs (respectively non-uniruled klt pairs) in dimensions at most $n$.
\end{lemma}

\begin{proof}
We only show the statement for non-uniruled pairs, as the other statement is analogous. 

Let $(X, \Delta)$ be a non-uniruled klt pair of dimension $k<n$ such that $K_X+\Delta$ is pseudoeffective. By passing to a small $\mathbb Q$-factorialization \cite[Theorem 10.5]{Fuj11}, we may assume that $X$ is $\mathbb Q$-factorial. By \cite[Lemma 2.16 and Theorem A]{LT19}, we may assume that $K_X+\Delta$ is nef. By Lemma \ref{lem:ratreal}, we may assume that $\Delta$ is rational. By \cite[Theorem 4.3]{GL13} it suffices to show that $\kappa(X,K_X+\Delta)=\nu(X,K_X+\Delta)$.

To this end, we borrow the idea from \cite[Lemma 3.2]{Has18}. By passing to a log resolution, we may assume that $(X, \Delta)$ is log smooth. Let $A$ be an abelian variety of dimension $n-k$ and set $Y:=X\times A$. Then $Y$ is not uniruled, since $A$ has no rational curves. If $\pi\colon Y\to X$ is the first projection, then $K_Y\sim \pi^*K_X$ and the pair $(Y, \pi^*\Delta)$ is a log smooth klt pair of dimension $n$ since $\pi$ is a smooth morphism. By assumption, $(Y,  \pi^*\Delta)$ has a good minimal model, hence $\kappa(Y,K_Y+\pi^*\Delta)=\nu(Y,K_Y+\pi^*\Delta)$. Since $K_Y+\pi^*\Delta\sim_\Q\pi^*(K_X+\Delta)$, we conclude.
\end{proof}

\begin{thm}\label{HasR}
The Nonvanishing conjecture for smooth projective varieties in dimension $n$ implies the Nonvanishing conjecture in dimensions at most $n$ and the existence of minimal models for log canonical pairs in dimensions at most $n$.
\end{thm}

\begin{proof}
The first part of the statement is contained in \cite[Theorem 1.4]{Has18}. For the second part, \cite[Theorem 1.4]{Has18} implies the existence of minimal models in the sense of Birkar--Shokurov for log canonical pairs in dimensions at most $n$. But then \cite[Theorem 1.7]{HH19} gives the existence of minimal models in the usual sense.
\end{proof}

\begin{lemma}\label{lem:exceptional}
Let $(X,\Delta)$ be a log canonical pair and let $ \varphi \colon (X,\Delta) \dashrightarrow (X',\Delta')$ be either a minimal model of $(X,\Delta)$, or a minimal model of $(X,\Delta)$ in the sense of Birkar--Shokurov, or a good minimal model of $(X,\Delta)$ in the sense of Birkar--Hashizume. Let $(p,q)\colon W\to X\times X'$ be a smooth resolution of indeterminacies of $\varphi$. Then there exist an effective $p$-exceptional divisor $E_p$ and an effective $q$-exceptional divisor $E_q$ such that
$$p^*(K_X+\Delta)+E_p\sim_\R q^*(K_{X'}+\Delta')+E_q.$$
In particular, $\kappa_\iota(X,K_X+\Delta)=\kappa_\iota(X',K_{X'}+\Delta')$ and $\nu(X,K_X+\Delta)=\nu(X',K_{X'}+\Delta')$.
\end{lemma}

\begin{proof}
We prove only the case when $(X',\Delta')$ is a minimal model of $(X,\Delta)$ in the sense of Birkar--Shokurov; the other cases are analogous.

Recall that $\Delta'=\varphi_*\Delta+E$, where $E$ is the reduced divisor containing all the $\varphi^{-1}$-exceptional prime divisors in its support. We have the ramification formulas 
$$K_W+p_*^{-1}\Delta\sim_\R p^*(K_X+\Delta)+E_1^p+E_1^+-E_1^-$$
and
$$K_W+q_*^{-1}\Delta'\sim_\R q^*(K_{X'}+\Delta')+E_2^q+E_2^+-E_2^-,$$ 
where $E_1^p$ is $p$-exceptional but not $q$-exceptional, $E_2^q$ is $q$-exceptional but not $p$-exceptional, $E_1^+,E_1^-,E_2^+$ and $E_2^-$ are effective and are both $p$-exceptional and $q$-exceptional, $E_1^+$ and $E_1^-$ have no common components, and $E_2^+$ and $E_2^-$ have no common components. This implies
\begin{align*}
p^*(K_X+\Delta)&+(E_1^p+q_*^{-1}E)+E_1^++E_2^-\\
&\sim_\R q^*(K_{X'}+\Delta')+(E_2^q+p_*^{-1}\Delta-q_*^{-1}\varphi_*\Delta)+E_2^++E_1^-.
\end{align*}
By the definition of minimal models in the sense of Birkar--Shokurov, we have $E_2^q+p_*^{-1}\Delta-q_*^{-1}\varphi_*\Delta\geq0$, and also $E_1^p+q_*^{-1}E\geq0$ since all the coefficients of $E_1^p$ are at least ${-}1$. This proves the result.
\end{proof}

Finally, we need the existence of dlt blowups, see \cite[Theorem 3.1]{KK10} and \cite[Theorem 10.5]{Fuj11}.

\begin{thm}\label{dlt}
Let $(X,\Delta)$ be a log canonical pair. Then there exists a $\Q$-factorial dlt pair $(Y,\Delta_Y)$, called a \emph{dlt model} or \emph{dlt blowup}\footnote{Dlt blowups in the literature occasionally satisfy some additional properties, which we will not need in this paper.} of $(X,\Delta)$, such that there exists a birational morphism $f\colon  Y\to X$ with $K_Y+\Delta_Y\sim_{\mathbb{R}}f^*(K_X+\Delta)$.
\end{thm}

\section{Nonvanishing for uniruled log canonical pairs}\label{sec:nonvan}

In this section, we prove our main result. 

\begin{proof}[Proof of Theorem \ref{main}]
By Theorem \ref{HasR} we may assume the Nonvanishing conjecture in dimensions up to $n-1$ and the existence of minimal models for log canonical pairs in dimensions up to $n-1$.

Let $(X, \Delta)$ be a uniruled log canonical pair of dimension $n$ such that $K_X+\Delta$ is pseudoeffective. By passing to a log resolution, we may assume that $(X, \Delta)$ is a log smooth pair. By \cite[Corollary 4.11]{Deb01} the divisor $K_X$ is not pseudoeffective.

In Steps 1--4 we assume that $\lfloor\Delta\rfloor=0$, so that the pair $(X,\Delta)$ is klt. We treat the remaining case when $\lfloor\Delta\rfloor\neq0$ in Steps 5--8.
	
	\medskip
	
	\emph{Step 1.}
In Steps 1--4 we assume that $\lfloor\Delta\rfloor=0$, so that the pair $(X,\Delta)$ is klt. For $\tau:=\tau(X, 0; \Delta)$ we have $0<\tau\leq1$. Then it suffices to show that $\kappa_{\iota}(X,K_X+\tau\Delta)\geq0$, hence by replacing $\Delta$ by $\tau\Delta$ we may assume that $\tau=1$. 

Analogously as in Step 1 of the proof of \cite[Theorem 3.1]{LT19}\footnote{In short, one runs an MMP for a pair $\big(X,(1-\zeta)\Delta\big)$, where $0<\zeta\ll1$ is a carefully chosen real number. This MMP terminates with a Mori fibre space $W\to Y$, and one replaces $(X,\Delta)$ by a suitable pair on a resolution of indeterminacies of the birational map $\rho\colon X\dashrightarrow W$. One may not replace $(X,\Delta)$ by $(W,\rho_*\Delta)$ since the map $\rho$ is not known to be $(K_X+\Delta)$-non-positive.}, one can show that we may assume the following:
	
	\medskip
	
	\emph{Assumption 1.}
	There exists a fibration $\xi\colon X\to Y$ to a normal projective variety $Y$ with $\dim Y<\dim X$ such that:
	\begin{enumerate}
		\item[(a$_1$)] $\nu(F, (K_X+\Delta)|_F)=0$ and $h^1(F, \mathcal{O}_F)=0$ for a very general fibre $F$ of $\xi$,
		\item[(b$_1$)] $K_X+(1-\varepsilon)\Delta$ is not $\xi$-pseudoeffective for any $\varepsilon>0$.
	\end{enumerate}

	\medskip
	
	\emph{Step 2.}
	If $\dim Y=0$, then $Y$ is a point and $\nu(X, K_X+\Delta)=0$. We conclude by \cite[Theorem 1.2]{Gon11}.
	
	\medskip
	
	\emph{Step 3.}
	Assume from now on that $\dim Y>0$. In this step we show that we may assume the following:
	
	\medskip
	
	\emph{Assumption 2.}
	There exists a fibration $g\colon X\to T$ to a normal projective variety $T$ such that:
	\begin{enumerate}
	\item[(a$_2$)] $g$ is a Mori fibre space given by a contraction of an extremal ray of the pair $\big(X, (1-\varepsilon)\Delta\big)$ for some $0<\varepsilon\ll 1$,
	\item[(b$_2$)] $K_X+\Delta\equiv_T0$.
	\end{enumerate}
\emph{However, instead of the pair $(X,\Delta)$ being log smooth, we may only assume that it is a $\Q$-factorial klt pair.}

\medskip

	To this end, we argue similarly as in Step 3 of the proof of \cite[Theorem 3.1]{LT19}. By (a$_1$) and by \cite[Corollary 2.18]{LT19} the divisor $K_X+\Delta$ is effective over $Y$; in particular, it has an NQC weak Zariski decomposition as in \cite[Definition 2.10]{LT19}. By assumptions stated at the beginning of the proof and by \cite[Theorem 4.4]{LT19} we may run a $(K_X+\Delta)$-MMP with scaling of an ample divisor over $Y$ which terminates, and we obtain a birational contraction $ \theta \colon X \dashrightarrow X' $. Set $\Delta':=\theta_*\Delta $ and let $\xi'\colon X'\to Y$ be the induced morphism. 
	
	By \cite[Lemma 2.19]{LT19} there exists a small rational number $\delta$ such that, if we run a $ \big(K_{X'}+(1-\delta)\Delta' \big) $-MMP with scaling of an ample divisor over $Y$, then this MMP is $(K_{X'}+\Delta')$-trivial. Note that $ K_{X'}+(1-\delta)\Delta'$ is not $\xi'$-pseudoeffective: indeed, by possibly choosing $\delta$ smaller, we may assume that the map $\theta$ is $ \big(K_X+(1-\delta)\Delta\big)$-negative, and the claim follows since $K_X+(1-\delta)\Delta$ is not $\xi$-pseudoeffective by (b$_1$). Therefore, this relative $ \big(K_{X'}+(1-\delta)\Delta' \big) $-MMP terminates with a Mori fibre space $f''\colon X''\to Y''$ over $Y$ by \cite[Corollary 1.3.3]{BCHM}. Let $\theta' \colon X' \dashrightarrow X''$ denote that MMP and set $\Delta'':=\theta'_*\Delta'$. 
	
	\begin{center}
		\begin{tikzcd}
			X \arrow[r, dashed, "\theta"] \arrow[dr, "\xi" swap] & X' \arrow[r, dashed, "\theta'"] \arrow[d, "\xi'"] & X'' \arrow[d, "f''"] \arrow[dl] \\
			& Y & Y'' \arrow[l]
		\end{tikzcd}
	\end{center} 
	
	Then the pair $ (X'',\Delta'') $ is $ \Q $-factorial klt and we have 
	$$K_{X''} + \Delta'' \equiv_{Y''} 0$$
	by \cite[Lemma 2.19]{LT19}. Therefore, by replacing $(X,\Delta)$ by $(X'',\Delta'')$ and by setting $T:=Y''$ and $g:=f''$, we achieve Assumption 2.
		
	\medskip
	
	\emph{Step 4.}
		By \cite[Theorem 0.2]{Amb05} and \cite[Theorem 3.1]{FG12} there exists an effective $\mathbb{R}$-divisor $\Delta_T$ on $T$ such that $(T, \Delta_T)$ is klt and $K_X+\Delta\sim_{\mathbb{R}}g^*(K_T+\Delta_T)$. Since $K_X+\Delta$ is pseudoeffective and $g$ is surjective, the divisor $K_T+\Delta_T$ is pseudoeffective. By assumptions in lower dimensions, we have $\kappa_{\iota}(T,K_T+\Delta_T)\geq0$, and hence $\kappa_{\iota}(X,K_X+\Delta)\geq0$.
		
		This concludes the proof if $\lfloor\Delta\rfloor=0$.

	\medskip
	
	\emph{Step 5.}
	From now on we assume that $\lfloor\Delta\rfloor\neq0$. For 
	$$\tau':=\max\big\{0,\tau(X, \Delta-\lfloor\Delta\rfloor; \lfloor\Delta\rfloor)\big\}$$ 
	we have $0\leq\tau'\leq1$. If $\tau'<1$, then $K_X+\Delta-\lfloor\Delta\rfloor+\tau'\lfloor\Delta\rfloor$ is klt and pseudoeffective, hence $\kappa_{\iota}(X,K_X+\Delta-\lfloor\Delta\rfloor+\tau'\lfloor\Delta\rfloor)\geq0$ by Steps 1--4, and so $\kappa_{\iota}(X,K_X+\Delta)\geq0$.	Thus, we may assume that $\tau'=1$.
	
	Analogously as in Step 1 of the proof of \cite[Theorem 3.1]{LT19}, by picking a decreasing sequence $ \left\{ \varepsilon_i \right\} $ of positive numbers such that $\varepsilon_i\to 0$ and by considering divisors $\Delta_i:=\Delta-\varepsilon_i\lfloor\Delta\rfloor$ instead of divisors $(1-\varepsilon_i)\Delta$, we show that we may assume the following:

\medskip

	\emph{Assumption 3.}
	There exists a fibration $\xi\colon X\to Y$ to a normal projective variety $Y$ with $\dim Y<\dim X$ such that:
	\begin{enumerate}
		\item[(a$_3$)] $\nu(F, (K_X+\Delta)|_F)=0$ and $h^1(F, \mathcal{O}_F)=0$ for a very general fibre $F$ of $\xi$,
		\item[(b$_3$)] $K_X+\Delta-\varepsilon\lfloor\Delta\rfloor$ is not $\xi$-pseudoeffective for any $\varepsilon>0$.
	\end{enumerate}
	
\medskip

\emph{Step 6.}
	If $\dim Y=0$, then we conclude as in Step 2 above.
	
	\medskip
	
	\emph{Step 7.}
		Assume from now on that $\dim Y>0$. Then as in Step 3 above, one can show that we may assume the following:
	
	\medskip
	
	\emph{Assumption 4.}
	There exists a fibration $g\colon X\to T$ to a normal projective variety $T$ such that:
	\begin{enumerate}
	\item[(a$_4$)] $g$ is a Mori fibre space given by a contraction of an extremal ray of the pair $\big(X, \Delta-\varepsilon\lfloor\Delta\rfloor\big)$ for some $0<\varepsilon\ll 1$,
	\item[(b$_4$)] $K_X+\Delta\equiv_T0$.
	\end{enumerate}
\emph{However, instead of the pair $(X,\Delta)$ being log smooth, we may only assume that it is a $\mathbb{Q}$-factorial log canonical pair and $\big(X, \Delta-\varepsilon\lfloor\Delta\rfloor\big)$ is klt.}

\medskip

	\emph{Step 8.}
By Assumption 4, we have that ${-}\big(K_X+\Delta-\varepsilon\lfloor\Delta\rfloor\big)$ is ample over $T$ and $K_{X}+\Delta \equiv_T 0$, thus $\lfloor\Delta\rfloor$ dominates $T$. By Theorem \ref{dlt}, there exists a dlt blowup $\pi\colon (Y, \Delta_Y)\to (X, \Delta)$, and set $g':=g\circ\pi$. Then $\lfloor\Delta_Y\rfloor$ dominates $T$, hence there exists an irreducible component $S$ of $\lfloor\Delta_Y\rfloor$ which dominates $T$. By \cite[Proposition 3.9.2]{Fuj07}\footnote{This result is stated in \cite{Fuj07} for $\Q$-divisors. The proof for $\R$-divisors is analogous, see for instance \cite[Proposition 2.8]{HanLi}.} there exists an effective $\mathbb{R}$-divisor $\Delta_S$ on $S$ such that $(S, \Delta_S)$ is dlt and $(K_Y+\Delta_Y)|_S\sim_{\mathbb{R}}K_S+\Delta_S$.

By Assumption 4, there exists an $\mathbb{R}$-Cartier $\mathbb{R}$-divisor $D$ on $T$ such that $K_Y+\Delta_Y\sim_{\mathbb{R}}g'^*D$. Therefore, $K_S+\Delta_S\sim_{\mathbb{R}}(g'|_S)^*D$. Since $K_Y+\Delta_Y$ is pseudoeffective, the divisor $D$ is pseudoeffective and hence $K_S+\Delta_S$ is pseudoeffective. By assumptions in lower dimensions, we have $\kappa_{\iota}(S,K_S+\Delta_S)\geq0$, and thus
$$\kappa_{\iota}(X,K_X+\Delta)=\kappa_{\iota}(Y,K_Y+\Delta_Y)=\kappa_{\iota}(T,D)=\kappa_{\iota}(S,K_S+\Delta_S)\geq0.$$
This concludes the proof.
\end{proof}

We immediately have:

\begin{proof}[Proof of Corollary \ref{uni4}]
The Nonvanishing conjecture for terminal threefolds was proved in \cite{Miy87,Miy88}; for a different proof, see \cite[Theorem 6.7 and Remark 6.8]{LP18}. Then the result follows from Theorem \ref{main}.
\end{proof}

As mentioned in the introduction, we have the following corollary.

\begin{coro}\label{equi}
The Nonvanishing conjecture in dimension $n-1$ is equivalent to the Nonvanishing conjecture for uniruled projective log canonical pairs in dimension $n$.
\end{coro}

\begin{proof}
One direction follows from Theorem \ref{main}. For the other direction, we only need to prove the Nonvanishing conjecture for smooth projective varieties in dimension $n-1$ by Theorem \ref{HasR}. 

Let $X$ be a smooth projective variety of dimension $n-1$ such that $K_X$ is pseudoeffective. Consider $Y:=X\times\mathbb{P}^1$. Then $\mathcal O_Y(K_Y)\simeq p^*\mathcal O_X(K_X)\otimes q^*\mathcal{O}_{\mathbb{P}^1}(-2)$, where $p$ and $q$ are the corresponding projections. Pick two different points $a$ and $b$ on $\mathbb{P}^1$. Then $K_Y+q^*a+q^*b\sim p^*K_X$ is pseudoeffective since $K_X$ is pseudoeffective. Since $Y$ is uniruled and of dimension $n$, and the pair $(Y,q^*a+q^*b)$ is dlt, we have
$$\kappa(X,K_X)=\kappa(Y, K_Y+q^*a+q^*b)\geq0$$
by assumption.
\end{proof}

\section{Around the existence of good minimal models}

In this section we prove Theorem \ref{main2} and several related results which are of independent interest. The following lemma is essentially a combination of main results of \cite{HH19}. Note, however, that it is not a special case of \cite[Theorem 1.2]{HH19} since we use different notions of good minimal models.

\begin{lemma}\label{HasHu}
Assume the existence of good minimal models for projective klt pairs with rational boundaries in dimension $n-1$. 

Let $(X,\Delta)$ be a projective log canonical pair of dimension $n$ such that $\kappa_{\iota}(X,K_X+\Delta)=\nu(X,K_X+\Delta)$. Then $(X,\Delta)$ has a good minimal model. If additionally $K_X+\Delta$ is nef, then it is semiample.
\end{lemma}

\begin{proof}
By Lemmas \ref{lem:ratreal} and \ref{ind}, we may assume the existence of good minimal models for klt pairs in dimensions at most $n-1$. 

By \cite[Theorem 1.2 and Theorem 1.7]{HH19}, we may run a $(K_X+\Delta)$-MMP with scaling of an ample divisor which terminates. Hence, we may assume that $K_X+\Delta$ is additionally nef, and we need to show that $K_X+\Delta$ is semiample. 

By \cite[Theorem 1.2]{HH19} the pair $(X,\Delta)$ has a good minimal model $(X',\Delta')$ in the sense of Birkar--Hashizume. Let $(p,q)\colon W\to X\times X'$ be a smooth resolution of indeterminacies of the map $\varphi\colon X\dashrightarrow X'$. Then by Lemma \ref{lem:exceptional} there exist an effective $p$-exceptional divisor $E_p$ and an effective $q$-exceptional divisor $E_q$ such that
$$p^*(K_X+\Delta)+E_p\sim_\R q^*(K_{X'}+\Delta')+E_q.$$
By \cite[Lemma 2.16]{GL13} and since $p^*(K_X+\Delta)$ and $q^*(K_{X'}+\Delta')$ are nef, we obtain
\begin{align*}
p^*(K_X+\Delta)&=P_\sigma\big(p^*(K_X+\Delta)+E_p\big)\\
&\sim_\R P_\sigma\big(q^*(K_{X'}+\Delta')+E_q\big)=q^*(K_{X'}+\Delta').
\end{align*}
This gives that $p^*(K_X+\Delta)$ is semiample, hence there exists a fibration $f\colon W\to Z$ and an ample $\R$-divisor $A$ on $Z$ such that $p^*(K_X+\Delta)\sim_\R f^*A$. If $C$ is any curve contracted by $p$, then $A\cdot f(C)=f^*A\cdot C=p^*(K_X+\Delta)\cdot C=0$, and therefore, $C$ is contracted by $f$. By the Rigidity lemma \cite[Lemma 1.15]{Deb01} the map $f$ factors through $p$, so $K_X+\Delta$ is $\R$-linearly equivalent to the pullback of $A$ to $X$, and thus is semiample.
\end{proof}

\begin{prop}\label{lem:tau}
Assume the existence of good minimal models for projective dlt pairs in dimension $n-1$. 

Let $(X,\Delta)$ be a pseudoeffective $\Q$-factorial projective dlt pair of dimension $n$ such that $\lfloor\Delta\rfloor\neq0$ and $K_X+\Delta-\varepsilon\lfloor\Delta\rfloor$ is not pseudoeffective for any $\varepsilon>0$. Then $\kappa_\iota(X,K_X+\Delta)=\nu(X,K_X+\Delta)$.
\end{prop}

\begin{proof}
	Analogously as in Steps 5--7 of the proof of Theorem \ref{main} one can show that we may assume the following:

	\medskip
	
	\emph{Assumption.}
	There exists a fibration $g\colon X\to T$ to a normal projective variety $T$ such that:
	\begin{enumerate}
	\item[(a)] $g$ is a Mori fibre space given by a contraction of an extremal ray of the pair $\big(X, \Delta-\varepsilon\lfloor\Delta\rfloor\big)$ for some $0<\varepsilon\ll 1$,
	\item[(b)] $K_X+\Delta\equiv_T0$.
	\end{enumerate}
\emph{However, instead of the pair $(X,\Delta)$ being $\mathbb{Q}$-factorial dlt, we may only assume that it is a $\mathbb{Q}$-factorial log canonical pair and $\big(X, \Delta-\varepsilon\lfloor\Delta\rfloor\big)$ is klt.}

\medskip

Now we use the notation from Step 8 of the proof of Theorem \ref{main}. Then that step shows that $\kappa_{\iota}(X,K_X+\Delta)=\kappa_{\iota}(S,K_S+\Delta_S)$, and one can show analogously that $\nu(X,K_X+\Delta)=\nu(S,K_S+\Delta_S)$. We conclude by our assumption in dimension $n-1$.
\end{proof}

We are now ready to prove:

\begin{proof}[Proof of Theorem \ref{main2}]
The proof is by induction on $n$, the claim being easy when $n=1$. By Lemmas \ref{lem:ratreal} and \ref{ind} we may assume the existence of good minimal models for non-uniruled klt pairs in dimensions at most $n-1$, hence by induction we may assume the existence of good minimal models for all log canonical pairs in dimensions at most $n-1$.

Let $(X, \Delta)$ be a log canonical pair of dimension $n$ such that $K_X+\Delta$ is pseudoeffective. By Lemma \ref{HasHu} it suffices to show that $\kappa_{\iota}(X, K_X+\Delta)=\nu(X, K_X+\Delta)$. 

By passing to a log resolution we may assume that $(X, \Delta)$ is log smooth. 

\medskip

\emph{Step 1.}
We assume first that $\lfloor\Delta\rfloor=0$. The pair $(X,\Delta)$ has a minimal model $(X',\Delta')$ by \cite[Theorem A]{LT19} and by assumption, since smooth varieties with the pseudoeffective canonical classes are not uniruled by \cite[Corollary 4.11]{Deb01}. By \cite[Proposition 3.2(3)]{Bir11} there exist finitely many $\Q$-divisors $\Delta_i$ and positive real numbers $r_i$ such that each pair $(X',\Delta_i)$ is klt, each $K_{X'}+\Delta_i$ is nef and $K_{X'}+\Delta'=\sum r_i(K_{X'}+\Delta_i)$. By \cite[Theorem 1.1]{DL15} and by assumption each $K_{X'}+\Delta_i$ is semiample, hence $K_{X'}+\Delta'$ is semiample.

\medskip

\emph{Step 2.}
From now on we assume that $\lfloor\Delta\rfloor\neq0$. By Proposition \ref{lem:tau} we may assume that there exists $0<\varepsilon<1$ such that $K_X+\Delta-\varepsilon\lfloor\Delta\rfloor$ is pseudoeffective.

We have that $\kappa_\iota\big(X, K_X+\Delta-\varepsilon\lfloor\Delta\rfloor\big)\geq0$ by Step 1 since $\big(X, \Delta-\varepsilon\lfloor\Delta\rfloor\big)$ is a klt pair. Therefore, there exists an $\mathbb{R}$-divisor $D\geq0$ such that $K_X+\Delta-\varepsilon\lfloor\Delta\rfloor\sim_\R D$. Pick $0<\delta<\varepsilon$. Then 
$$K_X+\Delta-(\varepsilon-\delta)\lfloor\Delta\rfloor\sim_\R D+\delta\lfloor\Delta\rfloor\quad\text{and}\quad K_X+\Delta\sim_\R D+\varepsilon\lfloor\Delta\rfloor.$$
Since $\big(X,\Delta-(\varepsilon-\delta)\lfloor\Delta\rfloor\big)$ is a klt pair and $K_X+\Delta-(\varepsilon-\delta)\lfloor\Delta\rfloor$ is pseudoeffective, the pair $\big(X,\Delta-(\varepsilon-\delta)\lfloor\Delta\rfloor\big)$ has a good minimal model by Step 1. Therefore, we have $\kappa_\iota\big(X,D+\delta\lfloor\Delta\rfloor\big)=\nu\big(X,D+\delta\lfloor\Delta\rfloor\big)$, so $\kappa_\iota\big(X,D+\varepsilon\lfloor\Delta\rfloor\big)=\nu\big(X,D+\varepsilon\lfloor\Delta\rfloor\big)$ by \cite[Lemma 2.9]{DL15}\footnote{This result is stated in \cite{DL15} for $\Q$-divisors, but the proof works verbatim for $\R$-divisors.}. This concludes the proof.
\end{proof}

An immediate corollary is the following.

\begin{coro}\label{rev}
The existence of good minimal models for uniruled projective log canonical pairs in dimension $n$ implies the existence of good minimal models for projective log canonical pairs in dimensions at most $n-1$.
\end{coro}

\begin{proof}
By Theorem \ref{main2} and Lemma \ref{ind} we only need to prove the existence of good minimal models for klt pairs with rational boundaries in dimension $n-1$. By assumption and by Corollary \ref{equi}, we may assume the Nonvanishing conjecture for log canonical pairs in dimension $n-1$.

Let $(X, \Delta)$ be a klt pair of dimension $n-1$ such that $K_X+\Delta$ is pseudoeffective and $\Delta$ is a $\Q$-divisor. By passing to a small $\mathbb Q$-factorialization \cite[Theorem 10.5]{Fuj11}, we may assume that $X$ is $\mathbb Q$-factorial. By \cite[Theorem 4.3]{GL13} it suffices to show that $\kappa(X,K_X+\Delta)=\nu(X,K_X+\Delta)$. By passing to a log resolution, we may assume $(X, \Delta)$ is log smooth. 

Consider $Y:=X\times\mathbb{P}^1$. Then $\mathcal O_Y(K_Y)\simeq p^*\mathcal O_X(K_X)\otimes q^*\mathcal{O}_{\mathbb{P}^1}(-2)$, where $p$ and $q$ are the corresponding projections. Pick two different points $a$ and $b$ on $\mathbb{P}^1$. Then $K_Y+p^*\Delta+q^*a+q^*b\sim_{\Q} p^*(K_X+\Delta)$ is pseudoeffective since $K_X+\Delta$ is pseudoeffective. Since $p$ is a smooth morphism, $(Y, p^*\Delta+q^*a+q^*b)$ is log smooth and thus a dlt pair. Since $Y$ is uniruled and of dimension $n$, $(Y, p^*\Delta+q^*a+q^*b)$ has a good minimal model by assumption and hence
\begin{align*}
\kappa(X,K_X+\Delta)&=\kappa(Y, K_Y+p^*\Delta+q^*a+q^*b)\\
&=\nu(Y, K_Y+p^*\Delta+q^*a+q^*b)=\nu(X,K_X+\Delta).
\end{align*}
This finishes the proof.
\end{proof}

As announced in the introduction, we also have:

\begin{thm}\label{gm2}
Assume the existence of good minimal models for non-uni\-ru\-led projective klt pairs with rational boundaries in dimension $n-1$. 

Then the Abundance conjecture for non-uniruled projective klt pairs with rational boundaries in dimension $n$ implies the Abundance conjecture for uniruled projective log canonical pairs in dimension $n$. 
\end{thm}

\begin{proof}
By Theorem \ref{main2} and Lemma \ref{ind} we may assume the existence of good minimal models for log canonical pairs in dimensions at most $n-1$. 

Now, let $(X, \Delta)$ be a uniruled log canonical pair of dimension $n$ such that $K_X+\Delta$ is nef. It suffices to show that $\kappa_{\iota}(X, K_X+\Delta)=\nu(X, K_X+\Delta)$ by Lemma \ref{HasHu}. By passing to a log resolution we may assume that $(X, \Delta)$ is log smooth, but we may no longer assume that $K_X+\Delta$ is nef. We are done as in Steps 1 and 2 of the proof of Theorem \ref{main2} and the only difference is that we replace \cite[Theorem A]{LT19} by \cite[Theorem C]{LT19} in Step 1 of that proof.
\end{proof}

Finally, we note the following corollary of Theorem \ref{main2} and of \cite{Hu16,HH19}:

\begin{coro}
Assume the existence of good minimal models for non-uni\-ru\-led projective klt pairs with rational boundaries in dimension $n-1$. 

Let $(X,\Delta)$ be a projective log canonical pair of dimension $n$ having a non-trivial morphism to an abelian variety. If $K_X+\Delta$ is pseudoeffective, then $(X,\Delta)$ has a good minimal model.
\end{coro}

\begin{proof}
By Theorem \ref{main2} and Lemma \ref{ind} we may assume the existence of good minimal models for log canonical pairs in dimensions at most $n-1$. 

Let $A$ be the abelian variety as in the statement. By \cite[Theorem 1.2 and Theorem 1.7]{HH19} we can run a $(K_X+\Delta)$-MMP over $A$ which terminates with a relative minimal model $(X',\Delta')$ of $(X,\Delta)$ over $A$. Let $\alpha\colon X'\to A$ be the induced morphism. If $K_{X'}+\Delta'$ were not nef, then there would exist the contraction $c_R\colon X\to Z$ of a $(K_{X'}+\Delta')$-negative extremal ray $R$. Then by the Cone theorem \cite[Theorem 1.1]{Fuj11}, $R$ is spanned by the class of some rational curve $C$ on $X'$. Since abelian varieties contain no rational curves, the curve $C$ has to be contracted by $\alpha$, a contradiction since $K_{X'}+\Delta'$ is $\alpha$-nef. Therefore, $K_{X'}+\Delta'$ is nef.

We now show that $K_{X'}+\Delta'$ is semiample. By Theorem \ref{dlt} we may assume that $(X',\Delta')$ is a $\Q$-factorial dlt pair. By \cite[Proposition 3.2(3)]{Bir11} there exist finitely many $\Q$-divisors $\Delta_i$ and positive real numbers $r_i$ such that each pair $(X',\Delta_i)$ is log canonical, each $K_{X'}+\Delta_i$ is nef and $K_{X'}+\Delta'=\sum r_i(K_{X'}+\Delta_i)$. By \cite[Theorem 1.1]{Hu16} we then have $\kappa(X, K_{X'}+\Delta_i)=\nu(X, K_{X'}+\Delta_i)$, hence each $K_{X'}+\Delta_i$ is semiample by Lemma \ref{HasHu}, and therefore so is $K_{X'}+\Delta'$.
\end{proof}

	\bibliographystyle{amsalpha}
	\bibliography{biblio}	

\end{document}